\documentclass[11pt]{amsart}
\usepackage{amssymb}
\usepackage{amssymb, amsmath}
\usepackage{verbatim}
\let\d=\partial

\let\wt=\widetilde

\def\cC{{\mathcal C}}

\def\cO{{\mathcal O}}
\def\cP{{\mathcal P}}
\def\cQ{{\mathcal Q}}

\def\R{{\mathbb R}}
\def\T{{\mathbb T}}

\def\virgp{\raise 2pt\hbox{,}}
\def\cdotpv{\raise 2pt\hbox{;}}

\def\div{ \hbox{\rm div}\,  }

\newcommand{\with}{\quad\hbox{with}\quad}
\newcommand{\andf}{\quad\hbox{and}\quad}

\newtheorem{thm}{Theorem}[section]

\newtheorem{prop}{Proposition}[section]

\newcommand{\ben}{\begin{eqnarray}}
\newcommand{\een}{\end{eqnarray}}
\newcommand{\beno}{\begin{eqnarray*}}
\newcommand{\eeno}{\end{eqnarray*}}

\begin{document}
\title[]{From compressible to  incompressible inhomogeneous flows
in the case of large data}
\author[R. Danchin]{Rapha\"{e}l Danchin}
\address[R. Danchin]{Universit\'{e} Paris-Est,  LAMA (UMR 8050), UPEMLV, UPEC, CNRS,
 61 avenue du G\'{e}n\'{e}ral de Gaulle, 94010 Cr\'{e}teil Cedex, France.} \email{raphael.danchin@u-pec.fr}
\author[P.B. Mucha]{Piotr Bogus\l aw Mucha}
\address[P.B. Mucha]{Instytut Matematyki Stosowanej i Mechaniki, Uniwersytet Warszawski, 
ul. Banacha 2,  02-097 Warszawa, Poland.} 
\email{p.mucha@mimuw.edu.pl}

\begin{abstract} This paper is concerned with the mathematical derivation of the inhomogeneous incompressible Navier-Stokes equations $(INS)$
from the compressible Navier-Stokes equations $(CNS)$ in the large volume viscosity limit. 
We first prove   a result of large time existence of regular solutions  for $(CNS).$  Next, as a consequence, 
we establish that  the solutions of $(CNS)$ converge to those of $(INS)$ when the volume viscosity tends to infinity. 
Analysis is performed in the two dimensional torus $\T^2,$ for general initial data.
In particular, we are able to handle \emph{large} variations of density.
\end{abstract}

\maketitle

\section{Introduction}

We are concerned with  the following compressible Navier-Stokes system:
\begin{equation}\label{CNS}
 \begin{array}{lcr}
  \rho_t + \div (\rho v) =0 & \mbox{ in } & (0,T)\times \T^2,\\[1ex]
  \rho v_t + \rho v \cdot \nabla v -\mu\Delta v - \nu \nabla \div v + \nabla P =0 
  & \mbox{ in } & (0,T)\times\T^2.
 \end{array}
\end{equation}

Above, the unknown nonnegative function  $\rho=\rho(t,x)$ and vector-field $v=v(t,x)$ stand for the  density and velocity of the fluid at $(t,x).$ 
The two real numbers  $\mu$ and $\nu$ denote the viscosity coefficients
and  are assumed to satisfy $\mu>0$ and $\nu+\mu>0.$  

We suppose that the  pressure function $P=P(\rho)$ is $C^1$ with $P'>0,$ and 
that $P(\bar\rho)=0$  for some  positive constant reference density $\bar\rho.$
 Throughout, we  set 
$$e(\rho):=\rho\int_{\bar\rho}^\rho\frac{P(t)}{t^2}\,dt.$$
Note that $e(\bar\rho)=e'(\bar\rho)=0$ and   that  $\rho e''(\rho)=P'(\rho).$ Hence  $e$ is a strictly convex function 
and,  for any interval $[\rho_*,\rho^*],$ there exist two constants $m_*$ and $m^*$ so that
\begin{equation}\label{eq:convexity}
m_*(\rho-\bar\rho)^2\leq e(\rho)\leq m^*(\rho-\bar\rho)^2. 
\end{equation}

The system is supplemented with the initial conditions
\begin{equation}
 v|_{t=0} = v_0\in\R^2\quad\hbox{and}\quad \rho|_{t=0} = \rho_0\in\R^+.
\end{equation}

 We aim at  comparing   the above compressible Navier-Stokes system with its incompressible but inhomogeneous
version. The system  in  question reads
\begin{equation}\label{INS}
\begin{array}{lcr}
 \eta_t + u\cdot \nabla\eta =0 & \mbox{ in } &  (0,T)\times \T^2, \\
 \eta u_t + \eta u\cdot \nabla u - \mu\Delta u + \nabla \Pi=0 & \mbox{ in } &  (0,T)\times \T^2, \\
 \div u=0 & \mbox{ in } & (0,T)\times \T^2.
\end{array}
\end{equation}

At the formal level, 
 one can expect the solutions to \eqref{CNS} to converge to those of \eqref{INS}
 when~$\nu$ goes to $+\infty.$  Indeed, the velocity equation of \eqref{CNS} may be rewritten
 $$
 \nabla\div v=\frac1\nu\Bigl(  \rho v_t + \rho v \cdot \nabla v -\mu\Delta v +\nabla P\Bigr)
 $$
 and thus $\nabla\div v$ should tend to $0$ when $\nu\to+\infty.$
 This means that $\div v$ should tend to be  independent of the space variable and, as it is the divergence of
 some periodic vector field, one must eventually have $\div v\to0.$
 As, on the other side,  one has for all value of $\nu,$ 
 $$\rho v_t + \rho v \cdot \nabla v -\mu\Delta v\quad\hbox{is a gradient},$$
this means that  if $(\rho,v)$ tends to some couple $(\eta,u)$ 
in  a sufficiently strong meaning, then necessarily $(\eta,u)$ should satisfy \eqref{INS}.  
\smallbreak
Hence, the question of finding an appropriate framework for justifying that heuristics naturally arises.
Let us first examine  the \emph{weak solution framework} as it requires
the minimal assumptions on the data. 
As regards System \eqref{CNS} with pressure law like $P(\rho)=a(\rho^\gamma-\bar\rho^\gamma)$ for some $a>0$ and $\gamma>1,$  the state-of-the-art for the  weak solution theory  is as follows    
(see \cite{Lions,NS} for more details):
\begin{thm}
Assume that the initial data $\rho_0$ and $v_0$ satisfy  $\sqrt\rho_0\,v_0 \in L_2(\T^2)$
and  $\rho_0 \in L_\gamma(\T^2).$ Then there exists a global in time weak solution  to \eqref{CNS}   such that
 \begin{equation}
  v\in L_\infty(\R_+;L_2(\T^2)) \cap L_2(\R_+;\dot H^1(\T^2))\ \hbox{ and } \  e(\rho) \in L_\infty(\R_+;L_1(\T^2))
 \end{equation}
 and, for all $T>0,$ 
 \begin{equation}\label{eq:energy}
 \int_{\T^2} \biggl(\frac12 \rho|v|^2 \!+\!e(\rho)\biggr)(T,\cdot)\,dx + \int_0^T (\mu\|\nabla v\|^2_{2} + 
 \nu \|\div v \|_{2}^2)\, dt \leq  \int_{\T^2} \biggl( \frac12\rho_0|v_0|^2 \!+\!e(\rho_0)\biggr)dx.
 \end{equation}
\end{thm}


For System \eqref{INS}, there is a similar  weak solution theory that has been initiated 
by  A. Kazhikhov in \cite{K}, then continued by J. Simon in \cite{S}
and completed by  P.-L. Lions in \cite{Lions96}.
However, to the best of our knowledge, it is not known how to connect System \eqref{CNS} to \eqref{INS} 
in  that framework.  Justifying the convergence in that setting may be extremely difficult 
owing to the fact that the key extra estimate for the density that allows to achieve 
the existence of weak solutions for  \eqref{CNS} strongly depends on the viscosity coefficient $\nu$, 
and collapses when  $\nu$ goes to infinity.


\smallbreak
This thus motivates us to consider the problem for more regular solutions.
As regards System \eqref{CNS} in the multi-dimensional case,  
recall that the global  existence issue of  strong  unique solutions 
has been  answered just partially, and mostly 
in the  small data case, see e.g. \cite{D00,Ko,MaNi,Mu03,MuZa1,MuZa2,VaZa}. 
For general large data (even if very smooth), only local-in-time solutions 
are available (see e.g. \cite{D01,Nash}).

 The theory of strong solution for the inhomogeneous Navier-Stokes system \eqref{INS} 
is more complete (see e.g. \cite{DaMu12,LaSo,HPZ,Li}).  In fact,  
 the results are roughly the same as for the homogeneous (that is with constant density)
 incompressible  Navier-Stokes system.
In particular, we  proved in \cite{DM17} that, in the two-dimensional case, system \eqref{INS} is uniquely 
and globally solvable in dimension two whenever the initial velocity is in $H^1$ and the initial density is nonnegative 
and bounded (initial data with vacuum may thus be considered). 

It is tempting to study whether those better properties  in dimension two for the (supposedly) 
limit system \eqref{INS} may help us to improve our knowledge of System \eqref{CNS}
in the case where the volume viscosity is very large. 
More precisely, we here want to  address the following two questions:
\begin{itemize}
\item For  regular data with no vacuum, then given any fixed $T>0,$ can we find $\nu_0$ so that 
the solution remains smooth (hence unique) until  time $T$ for all~$\nu \geq \nu_0$~? 
\item Considering a family $(\rho_\nu,v_\nu)$ of solutions to \eqref{CNS} and letting $\nu \to \infty,$ can we 
show strong convergence to some couple $(\eta,u)$ satisfying \eqref{INS} and, as the case may be, give
an upper bound for the rate of convergence~?
\end{itemize}

Those two issues have been considered recently in our paper \cite{DM-AM}, 
 in the particular case where the initial density is a perturbation of order  $\nu^{-\frac12}$
of some  constant positive density (hence the limit system is just 
 the classical incompressible Navier-Stokes equation).  There, our results were based 
 on Fourier analysis and involved so-called critical Besov norms. 
 The cornerstone of the method was a refined analysis 
 of the linearized system about the constant state $(\rho,v)=(\bar\rho,0),$
 thus precluding us from considering large density variations.

The present  paper aims at shedding  a new light  on this issue, pointing out
different results and techniques  than in \cite{DM-AM}. In particular, we will go beyond the slightly 
inhomogeneous case, and will be able to consider  large variations of density.
As regards the techniques, we here meet  another  motivation for our paper, which  is strictly mathematical:
 we want to advertize  two tools, that  can be  of some use in  the analysis  of systems of fluid mechanics:
\begin{itemize}
\item   The first one is a nonstandard estimate with (limited) 
loss of integrability for solutions of   the transport equation 
by a non Lipschitz vector-field that has been first pointed out by B. Desjardins in \cite{Des-DIE}
(see Section \ref{s:transport}). Proving it requires some Moser-Trudinger 
inequality that holds true only in dimension two\footnote{Consequently, we do not know how
to adapt our approach to the higher dimensional case.}.
\item  The second tool is an estimate for a parabolic system with just bounded coefficients in the maximal regularity framework  of  $L_p$ spaces with $p$ close to 2 (Section \ref{s:maxreg}). 
\end{itemize}
\medbreak
For notational simplicity, \emph{we assume from now on that the shear viscosity $\mu$ is equal to~$1$}
(which may always been achieved after a suitable rescaling). 
Our answer to the first question then reads as follows:
 \begin{thm}\label{th:1} Fix some $T>0.$ Let  $\rho_*$ and $\rho^*$
 satisfy $0<\rho_*<\rho^*,$ and assume that 
 \begin{equation}\label{eq:rho0}
 2\rho_* \leq \rho_0\leq \frac 12\, \rho^*.\end{equation}  
  There exists an exponent $q>2$ depending only on $\rho_*$ and $\rho^*$ 
  such that if   $\nabla\rho_0\in L_q(\T^2)$ then for any  vector 
  field  $v_0$ in $W^{2-2/q}_{q}(\T^2)$  
 satisfying \begin{equation}\label{initial-div}\nu^{1/2}\|\div v_0\|_{L_2} \leq 1,\end{equation}
   there exists $\nu_0=\nu_0(T,\rho_*,\rho^*,\|\nabla\rho_0\|_q,\|v_0\|_{W^{2-\frac2q}_q},P,q)$ such that
      System \eqref{CNS} with  $\nu \geq \nu_0$  has  a unique  solution $(\rho,v)$   on the time interval $[0,T],$ fulfilling 
  \begin{equation}
   v\in \cC([0,T];  W^{2-2/q}_{q}(\T^2)),\quad v_t,\nabla^2 v\in L_q([0,T]\times\T^2),\quad \rho \in \cC([0,T];W^1_q(\T^2)),
  \end{equation}
  \begin{equation}
 \andf  \rho_* \leq \rho(t,x) \leq \rho^* \ \mbox{ for all }\ (t,x) \in  [0,T]\times\T^2.\hspace{5cm}
  \end{equation}

  Furthermore, there exists a constant $C_q$ depending only on $q,$ 
  a constant $C_P$ depending only on $P,$
  and a universal constant $C$ such that for all $t\in[0,T],$
  \begin{multline}\label{E0}
  \|v(t)\|_{H^1}+\nu^{\frac12}\|\div v(t)\|_{L_2}+\|\rho(t)-\bar\rho\|_{L_2}+\|\nabla v\|_{L_2([0,t];H^1)} +\|v_t\|_{L_2(0,t\times\R^2)}
  \\+ \nu^{\frac12}\|\nabla\div v\|_{L_2(0,t\times\R^2)}\leq C e^{C\|v_0\|_{2}^4} E_0,
  \end{multline}
  \begin{multline}\label{E1}
  \|v(t)\|_{W^{2-\frac2q}_{q}}\!+\! \|v_t,\nabla^2v,\nu\nabla\div v\|_{L_q([0,t]\times\T^2)} 
  \\\leq  C_q\Bigl(\|v_0\|_{W^{2-\frac2q}_q}+C_Pt^{\frac1q}\bigl(1+\|\nabla\rho_0\|_{L_q}\bigr)
    \exp (t^{\frac1{q'}} I_0(t))\Bigr)\end{multline}
\begin{equation}\label{E2}
 \hbox{and}\quad\|\nabla \rho(t)\|_{L_q}\leq \bigl(1+\|\nabla\rho_0\|_{L_q}\bigr)  \exp\bigl\{t^{\frac1{q'}}I_0(t)\bigr\},\hspace{4cm}
\end{equation}
with $ E_0:=1+
\|v_0\|_{H^1}+ \|\rho_0-\bar\rho\|_{L_2}$ 
and
$$I_0(t):= C_q\biggl(\|v_0\|_{W^{2-\frac2q}_{q}}+C_Pt^{\frac1q}\bigl(1+\|\nabla\rho_0\|_{L_q}\bigr)e^{CE_0^2te^{C\|v_0\|_{L_2}^4}}\biggr)\cdotp$$
 \end{thm}

 As the data we here consider are regular and bounded away from zero, the short-time  existence
 and uniqueness issues are clear (one may e.g. adapt \cite{D10} to the case of periodic boundary conditions). 
 In order to achieve large time existence, we shall first take
 advantage of a rather standard higher order energy estimate (at the $H^1$ level for the velocity)
 that will provide us with a control of $\nabla v$ in $L^2(0,T;H^1)$
 in terms of the data and of the norm of $\nabla\rho$ in $L^\infty(0,T;L^2).$
 The difficulty now is to control that latter norm, given that, at this stage, one has no
  bound for $\nabla v$ in $L^1(0,T;L^\infty).$
 It may be overcome by adapting to our framework some estimates with 
 loss of integrability for the transport equation, that have been first pointed 
 out by B. Desjardins in \cite{Des-DIE}.
 However, this is not quite the end of the story since those estimates involve the quantity
 $\int_0^T\|\div v\|_{L^\infty}\,dt.$  Then, the key observation is that
 the linear maximal regularity theory  for the linearization of the momentum equation 
 of \eqref{CNS} (neglecting the pressure term and taking $\rho\equiv1$) provides,  for all $1<q<\infty,$ 
 a control on $\nu\|\nabla\div v\|_{L_q(0,T;L_q(\T^2))}$   (not just $\|\nabla\div v\|_{L_q(0,T;L_q(\T^2))}$)
in terms of $\|v_0\|_{W^{2-\frac2q}_q}.$
  In our framework where $\rho$ is not constant, it turns out 
 to be possible to recover a similar estimate if $q$ is close enough to $2,$ 
 and thus to eventually have, by Sobolev embedding,  $\int_0^T\|\div v\|_{L^\infty}\,dt=\cO(\nu^{-1}).$
 Then, putting all the arguments together and bootstrapping 
 allows to get  all the estimates of Theorem \ref{th:1}, for large enough $\nu.$  
 \medbreak
 Regarding the  asymptotics  $\nu \to +\infty,$  it is clear that  if one starts with fixed initial 
 data, then uniform estimates are available  from Theorem \ref{th:1}, only if
 we assume that  $\div v_0\equiv0.$
 Under that assumption,  Inequalities \eqref{E0} and \eqref{E1} already ensure  that
 $$
 \div v=\cO(\nu^{-1/2})\ \hbox{ in }\  L_\infty(0,T;L_2)\quad\hbox{and}\quad
\nabla\div v=\cO(\nu^{-1})\ \hbox { in }\ L_q(0,T\times\T^2).
 $$
 Then, combining with  the uniform bounds provided by \eqref{E1} and \eqref{E2}, it is not difficult
 to pass to the weak limit in System \eqref{CNS} and to find
 that the limit solution fulfills System~\eqref{INS}. 
 \smallbreak
 In the theorem below, we state a   result that involves strong norms of all quantities
 at the level of energy norm,   and exhibit an explicit  rate of convergence.
 \begin{thm}\label{th:2}
  Fix some  $T>0$ and take initial data  $(\rho_0,v_0)$ fulfilling the assumptions of Theorem \ref{th:1} with, 
  in addition, $\div v_0\equiv0.$ Denote by $(\rho_\nu,v_\nu)$ the corresponding solution
  of \eqref{CNS} with volume viscosity $\nu\geq\nu_0.$
  Finally,  let  $(\eta,u)$ be the global solution of \eqref{INS} supplemented
  with the same initial data $(\rho_0,v_0).$
    Then we have
\begin{multline}\label{eq:conv}
 \sup_{t\leq T} \bigl(\|\rho_\nu(t)-\eta(t)\|_{L_2}^2 + \|\cP v_\nu(t) - u(t)\|_{L_2}^2 + \|\nabla\cQ v_\nu(t)\|_{L_2}^2\bigr)\\+ \int_0^T\bigl(\|\nabla (\cP v_\nu -u)\|^2_{L_2} + \|\nabla\cQ v_\nu\|^2_{H^1}\bigr) dt \leq C_{0,T} \nu^{-1},
\end{multline}
 where $\cP$ and $\cQ$  are the Helmholtz projectors 
 on divergence-free and potential  vector fields, respectively\footnote{that are defined by 
$\cQ v := -\nabla (-\Delta)^{-1} \div v \mbox{ \  and \ } \cP v := v - \cQ v.$}, and 
 where $C_{0,T}$ depends only on $T$ and on the norms of the initial data.
 \end{thm}

At first glance, one may think our issue to be closely  related to the  question of low Mach number limit studied in  e.g. \cite{D02,FeNo}. However,  there is an essential difference in the 
mechanism leading to convergence as may be easily seen from a rough 
analysis of the linearized system \eqref{CNS}. 
Indeed, in the case $\bar\rho=\mu=1$ and $P'(1)=1$ (for notational simplicity),
 that linearization reads  (in the unforced case):
\begin{equation*}
 \begin{array}{l}
  \eta_t +\div u=0, \\[1ex]
v_t -\Delta v - \nu \nabla \div v +  \nabla \eta =0.
 \end{array}
\end{equation*}
Eliminating the velocity we obtain the damped wave equation
$$
 \eta_{tt} - (1+\nu)\Delta \eta_t -  \Delta \eta=0,
$$
that  can be solved explicitly at the level of the Fourier transform. We obtain two modes, one strongly parabolic, disappearing for $\nu \to \infty$, and the second one having the following 
form, in the high frequency regime:
\begin{equation*}\label{formal}
  \eta (t) \sim  \eta(0) e^{-\frac{t}{(1+\nu)}} \to  \eta(0).
\end{equation*}
This means that   at the same time, we have that  $\eta (t)$ tends strongly 
 to   $0$ as $t \to+ \infty$ even for very large $\nu,$  
 but that for all $t>0$ (even very large), $\eta(t)\to\eta(0)$ when $\nu$ tends to $+\infty.$
  \medbreak
The behavior corresponding to the low Mach number limit is  of a different nature,  
as it corresponds to the linearization
\begin{equation*}
 \begin{array}{l}
  \eta_t +\frac1\varepsilon\div u=0, \\[1ex]
v_t -\Delta v - \nu\nabla \div v +  \frac1\varepsilon\nabla \eta =0,
 \end{array}
\end{equation*}
which leads to the wave equation
$$ \eta_{tt} - (1+\nu)\Delta \eta_t -  \frac1{\varepsilon^2}\Delta \eta=0.$$
Asymptotically for $\varepsilon\to0,$ the above  damped wave equation
behaves as a wave equation with propagation speed $1/\varepsilon.$
Hence, in the periodic setting, we have huge oscillations that preclude
any strong convergence result. However, after filtering 
by the wave operator, convergence becomes strong, 
which entails  weak convergence, back to the original unknowns (see \cite{D02} for more details).  
\smallbreak 
The main idea of Theorem \ref{th:2} is just to compute the distance
between the compressible and the incompressible solutions,  by means of the standard energy norm
(in sharp  contrast with the approach in \cite{DM-AM} where critical Besov norms are used).
In order to do so, it is convenient to decompose $\rho-\eta$ into  two parts:
 $$
 \rho-\eta=(\rho-\wt\rho)+ (\wt\rho-\eta)$$ where the auxiliary density $\wt\rho$ 
is the transported of  $\rho_0$ by the flow of the divergence-free vector-field $\cP v.$
As  the bounds of Theorem \ref{th:1} readily ensure that $\|\rho-\wt\rho\|_q=\cO(\nu^{-1}),$
one may, somehow, perform the energy argument  as if 
comparing $(\wt\rho,v)$ and $(\eta,u).$
\medbreak
We end that introductory part presenting the main notations 
that are used throughout the paper.   By $\nabla$ we denote the gradient with respect to space variables, and by  $u_t,$ the time derivative of function $u$.
   By $\|\cdot\|_{L_p(Q)}$ (or sometimes just $\|\cdot\|_p$), 
    we mean the $p$-power Lebesgue norm corresponding to  the set $Q,$ and $L_p(Q)$ is 
   the corresponding Lebesgue space. We denote  by $H^s$ and $W^s_p$  the Sobolev
   (Slobodeckij for $s$ not integer) space on the torus $\T^2,$   and put $H^s=W^s_2.$
   The homogeneous versions of those spaces (that is the corresponding
   subspace of functions with null mean) are  denoted by $\dot H^s$ and $\dot W^s_p,$ 
respectively.


Generic constants are denoted by $C,$  $A\lesssim B$ means that $A\leq CB,$
and $A\approx B$ stands for $C^{-1}A\leq B\leq CA.$


\section{Energy estimates}\label{s:energy}

The aim of this part is to provide bounds via energy type estimates. We assume that the density is bounded from  above and below. 
Let us first recall the basic energy identity.
\begin{prop}
 For any $T>0,$  sufficiently smooth solutions to \eqref{CNS}   obey Inequality \eqref{eq:energy}.
\end{prop}
\begin{proof}
That fundamental estimate follows from testing the momentum equation by $v$ and integrating
by parts in the diffusion and pressure terms. Indeed: using  the 
definition of $e$ and the mass equation, we get
$$\begin{aligned}
\int_{\T^2}\nabla P\cdot v\,dx&=\int_{\T^2}\frac{P'(\rho)}\rho\nabla\rho\cdot(\rho v)\,dx=\int_{\T^2}\nabla(e'(\rho))\cdot(\rho v)\,dx
\\&=-\int_{\T^2} e'(\rho)\,\div(\rho v)\,dx
=\int_{\T^2} e'(\rho)\rho_t\,dx=\frac d{dt}\int_{\T^2} e(\rho)\,dx.\end{aligned}
$$
Then integrating in time completes the proof. 
\end{proof}








Let us next derive a higher order energy estimate, pointing 
out the dependency with respect to the volume viscosity $\nu.$
\begin{prop}\label{p:E2}
 Assume that there exist positive constants $\rho_*<\rho^*$  such that 
 \begin{equation}\label{eq:nonvacuum}
 \rho_* \leq \rho(t,x) \leq \rho^*  \mbox{ \ \ for all \ } (t,x) \in [0,T] \times  \T^2 .
\end{equation}
Then solutions to \eqref{CNS} with $\mu=1$ fulfill the following inequality: 
\begin{multline}\label{eq:e0}
 \|v(T),\nabla v(T), \rho(T) -\bar\rho\|_{2}^2 + \nu \|\div v(T)\|_{2}^2 + 
 \int_0^T (\|\nabla^2 v,\nabla v,v_t\|_{2}^2 + \nu \|\div v\|_{H^1}^2 )\,dt \\
 \leq C \exp\Bigl(C\|v_0\|_{2}^4\Bigr) \biggl(  \|v_0,\nabla v_0, \rho_0 -\bar\rho\|_{2}^2 
 + \nu \|\div v_0\|_{2}^2 + \nu^{-1}T\|v_0\|_{2}^2
+  \nu^{-1}\int_0^T \|\nabla \rho \|_{2}^2\,dt\biggr),
\end{multline}
provided $\nu$ is larger than some $\nu_0=\nu_0(\rho_*,\rho^*,P).$  
\end{prop}
\begin{proof} We take the $\T^2$ inner product of  the momentum equation with  $v_t,$ getting
\begin{equation}\label{eq:e1}
 \int_{\T^2}  \rho |v_t|^2 \,dx + \frac{1}{2}\frac{d}{dt} \int_{\T^2} \bigl(|\nabla v|^2 + \nu (\div v)^2\bigr) dx +
 \int_{\T^2} \nabla P \cdot v_t\, dx = - \int_{\T^2} (\rho v \cdot \nabla v)\cdot v_t\, dx.
\end{equation}

Integrating by parts and using the mass equation yields
$$\begin{aligned}
 \int_{\T^2} \nabla P\cdot v_t \, dx= -\int_{\T^2} P\,\div v_t\, dx &= -\frac{d}{dt} \int_{\T^2} P\,\div v\, dx + \int_{\T^2} P'(\rho) \rho_t \div v\, dx\\
&= -\frac{d}{dt} \int_{\T^2} P\,\div v\, dx - \int_{\T^2} P'(\rho) \,\div(\rho v)\, \div v\, dx.\end{aligned}
$$
Hence putting together with \eqref{eq:e1}, 
\begin{multline}\label{eq:e2}
 \frac{1}{2} \frac{d}{dt} \int_{\T^2} (|\nabla v|^2 + \nu (\div v)^2 - 2P\div v)\,dx + \int_{\T^2} \rho |v_t|^2 \,dx \\= 
\int_{\T^2} P'(\rho) \div (\rho v) \div v\, dx - \int_{\T^2} (\rho v\cdot \nabla v)\cdot v_t\,dx.
\end{multline}
Now, setting $K(\rho)=\rho P'(\rho)-P(\rho),$ one can check that
$$\begin{aligned}
\int_{\T^2} P'(\rho) \div (\rho v) \div v\, dx&=\int_{\T^2}(\div v)\;v\cdot\nabla(P(\rho))\,dx
+\int_{\T^2} \rho\nabla P'(\rho)\,(\div v)^2\,dx\\
&=-\int_{\T^2} P(\rho) \,v\cdot\nabla\div v\,dx +\int_{\T^2} K(\rho) (\div v)^2\,dx.\end{aligned}
$$
Hence, if \eqref{eq:nonvacuum} is fulfilled then we have
\begin{multline}\label{eq:e3}
 \frac{d}{dt} \int_{\T^2} \Bigl(|\nabla v|^2 + \nu (\div v)^2 - 2P(\rho)\div v\Bigr)dx + \int_{\T^2} \rho |v_t|^2\,dx \\\leq 
C\int_{\T^2} \bigl(|v\cdot\nabla \div v| + (\div v)^2 + |v\cdot \nabla v|^2 \bigr)dx .
\end{multline}
Next, testing the momentum equation by $\Delta v$ we get
\begin{equation*}\label{eq:e4}
\int_{\T^2} \bigl(|\Delta v|^2 + \nu |\nabla \div v |^2\bigr) dx - \int_{\T^2} \rho v_t \cdot\Delta v\, dx - 
\int_{\T^2} \nabla P \cdot  \Delta v \,dx \leq 
\int_{\T^2} |\rho v \cdot \nabla v \Delta v|\,dx.
\end{equation*}
Note that
$$- \int_{\T^2} \nabla P \cdot \Delta v\, dx= - \int_{\T^2} \nabla P \cdot \nabla\div v\, dx\leq C\int_{\T^2} |\nabla \rho| |\nabla \div v|\, dx.$$
Then,  combining with the basic energy identity and with \eqref{eq:e3} and introducing
\begin{equation}\label{eq:e5}
 E(v,\rho):=\int_{\T^2} \biggl(\rho |v|^2+2e(\rho) + |\nabla v|^2 +\nu\,( \div v)^2 -2 P(\rho)\div v\biggr) dx,
\end{equation}
we find,  
\begin{multline}\label{eq:e6}
 \frac{d}{dt} E(v,\rho)+
 \int_{\T^2} \rho|v_t|^2\, dx +\frac1{\rho^*}\int_{\T^2}\bigl(|\nabla v|^2+|\nabla^2 v|^2+\nu(\div v)^2+\nu |\nabla \div v|^2\bigr)dx \\ \leq \int_{\T^2} |v_t \cdot\Delta v|\, dx 
+ C\int_{\T^2} \bigl((\div v)^2+|v\cdot \nabla \div v| + \rho|v\cdot\nabla v|^2 + \frac1{\rho^*}|\nabla \rho|  |\nabla \div v| \bigr)dx\cdotp
\end{multline}
Hence, denoting
\begin{equation*}
 D(v):=\|\nabla v\|_{H^1}^2 +\|\sqrt\rho\,v_t\|^2_{L_2} + \nu \|\div v\|_{H^1}^2,
\end{equation*}
 Inequality \eqref{eq:e6} implies that for large  enough $\nu,$
$$ \frac{d}{dt}E(v,\rho) + \frac1{\rho^*}D(v) \leq C \int_{\T^2} \bigl(|v|^2|\nabla v|^2 
 + (|v|+ |\nabla \rho|)|\nabla \div v|\bigr) dx.$$
Of course, from the Ladyzhenskaya inequality, we have 
$$
\int_{\T^2}|v\cdot\nabla v|^2\,dx \leq C\|v\|_{2}\|\nabla v\|_{2}^2\|\Delta v\|_{2}.
$$
Therefore, we end up with
\begin{equation*}
 \frac{d}{dt}E(v,\rho) +  \frac1{\rho^*}D(v) \leq 
C\bigl(\|v\|^2_{2}\|\nabla v\|^2_{2}\|\nabla v\|^2_{2} + \nu^{-1}(\|v\|^2_{2} + \|\nabla \rho\|^2_{2})\bigr)\cdotp
\end{equation*}
Let us notice that  if $\nu\geq\nu_0(\rho_*,\rho^*,P)$ then we have, according to \eqref{eq:convexity}, 
 \begin{equation}\label{eq:e7}
 E(v,\rho)\approx \|v\|_{H^1}^2+ \| \rho -\bar\rho\|^2_{L_2} + \nu \|\div v\|^2_{L_2}. 
 \end{equation}
Hence  Gronwall inequality yields 
$$
\displaylines{
 E(v(T),\rho(T))+ \frac1{\rho^*}\int_0^TD(t)\,dt\leq \exp\biggl(C\int_0^T \|v\|^2_{2} \|\nabla v\|^2_{2} \,dt\biggr)\hfill\cr\hfill\times
  \biggl(E(v_0,\rho_0) + \frac C\nu\int_0^T  \exp\biggl(-C\int_0^t \|v\|^2_{2} \|\nabla v\|^2_{2} \,dt\biggr)  
 \bigl(\|v\|^2_{2} +\|\nabla \rho\|_{2}^2\bigr)dt\biggr)\cdotp}
$$
Remembering that the basic energy inequality implies that
$$\int_0^T \|v\|^2_{2} \|\nabla v\|^2_{2} \,dt\leq C\|v_0\|_{2}^4,$$
one may conclude that 
$$
  E(v(T),\rho(T)) + \frac1{\rho^*} \int_0^T D(v)\, dt
 \leq \exp\Bigl(C\|v_0\|_{2}^4\Bigr)  \biggl(E(v_0,\rho_0) + \frac C\nu\biggl(\|v_0\|_{2}^2\,T+\int_0^T\|\nabla\rho\|_{2}^2\,dt\biggr)
 \biggr),
 $$
 which obviously yields \eqref{eq:e0}. 
 \end{proof}


\section{Estimates with loss of integrability for the transport equation}\label{s:transport}

We are concerned with the proof of regularity estimates for the   transport equation
\begin{equation}\label{eq:T}
 \rho_t + v \cdot \nabla \rho + \rho \, \div v=0
\end{equation}
in some endpoint case where the transport field $v$ fails to be in $L_1(0,T;{\rm Lip})$ by a little.
\medbreak
More exactly, we aim at extending  Desjardins' results  in \cite{Des-DIE} to non divergence-free transport fields.  
Our  main result reads:
\begin{prop}\label{thm:transport}
 Let $1\leq q\leq \infty$ and $T>0$. Let $\rho_0 \in W^1_q(\T^2)$ and $v\in L_2(0,T;H^2(\T^2))$ such that $\div v \in L_1(0,T;L_\infty(\T^2)) \cap L_1(0,T;W^1_q(\T^2)).$
 Then the solution to \eqref{eq:T} fulfills for all $1\leq p<q,$ 
 $$\displaylines{
\sup_{t<T}  \|\nabla \rho(t)\|_{p}\leq 
K\biggl(\|\nabla \rho_0\|_{q}+\|\rho_0\|_{\infty}\sup_{t<T}\Big\|\int_0^t \nabla \div v\,d\tau\Big\|_{q}\biggr)\times \hfill\cr\hfill\times
    \exp\Big\{CT\int_0^T \|\nabla^2 v\|_{2}^2 \,dt\Big\}
       \exp\Big\{\int_0^T \|\div v\|_{\infty} \,dt\Big\},\quad}$$
where $K$ is an absolute constant, and where the constant $C$ depends only on $p$ and $q.$
\end{prop}

\begin{proof} 
We proceed by means of the standard characteristics method: our assumptions guarantee 
that $v$ admits a unique (generalized) flow $X$, solution to 
\begin{equation}\label{eq:T1}
 X(t,y)=y+\int_0^tv(\tau,X(\tau,y))\,d\tau.
\end{equation}

Then, setting 
\begin{equation}\label{eq:T2}
 u(t,y):=v(t,X(t,y))\andf a(t,y)=\rho(t,X(t,y)),
\end{equation}
  equation \eqref{eq:T} recasts as follows:
\begin{equation}\label{eq:T3}
 \frac{d a(t,y)}{dt} = -(\div v)(t,X(t,y)) \cdot a(t,y),
\end{equation}
the unique solution of which is given by
\begin{equation}\label{eq:T4}
 a(t,y)=\exp\bigg\{ \!-\!\int_0^t (\div v)(\tau,X(\tau,y))\, d\tau\bigg\} a_0(y).
\end{equation}

{}From the chain rule and Leibniz formula, we thus infer
$$\displaylines{\quad
 \nabla_y a(t,y) = \exp\Big\{-\!\int_0^t \!(\div v)(\tau,X(\tau,y))\, d\tau\!\Big\} \biggl(\nabla_y a_0(y)
\hfill\cr\hfill - a_0(y) \int_0^t (\nabla\div v)(\tau,X(\tau,y))\cdot\nabla_yX(\tau,y)\, d\tau\biggr)\cdotp\quad}
$$

Our goal  is  to estimate all these quantities in the Eulerian coordinates. Note that by \eqref{eq:T1}
and Gronwall lemma,  we obtain point-wisely that, denoting $Y(t,\cdot):=(X(t,\cdot))^{-1},$
$$ |\nabla_y X(t,y)| \!\leq\! \exp\Big\{ \!\int_0^t \!|\nabla_x v(\tau,X(\tau,y))| \,d\tau\!\Big\}\!\!\quad\hbox{and}\quad\!\!  |\nabla_x Y(t,x)| \!\leq\! \exp\Big\{ \!\int_0^t\! |\nabla_y u(\tau,Y(\tau,x))| \,d\tau\!\Big\} \cdotp
$$
As   $\nabla_x\rho(t,x) =\nabla_y a(t,Y(t,x))\cdot\nabla_xY(t,x),$ we get 
\begin{multline}\label{eq:T7}
 |\nabla\rho (t,x)|\leq   
 \exp\Big\{ 3\int_0^t |\nabla  v(\tau,X(\tau, Y(t,x)))|  \,d\tau \Big\} \\
\times  \biggl(|\nabla \rho_0(Y(t,x))|+|\rho_0(Y(t,x))|\bigg| \int_0^t \nabla\div v(\tau,X(\tau, Y(t,x))\,d\tau\bigg|\biggr)\cdotp
\end{multline}

Recall that the Jacobian of the change of coordinates $(t,y)\to (t,x)$ is given by
\begin{equation}\label{eq:T8}
 J_X(t,y)=\exp\Big\{ \int_0^t \div v(\tau,X(\tau,y)) \,d\tau\Big\} \leq \exp\Big\{ \int_0^t 
 \|\div v\|_{\infty}\, d\tau\Big\}\cdotp
\end{equation}

Hence taking the $L_p(\T^2)$ norm and using   H\"older inequality 
with $\frac 1p = \frac 1q + \frac 1m,$ we get
\begin{multline}\label{eq:T10}
 \|\nabla \rho(t)\|_{p} \leq
  \exp\Big\{ \frac1q\int_0^t  \|\div v\|_{\infty}\, d\tau\Big\}\  \biggl(\|\nabla \rho_0\|_{q} \\+ \|\rho_0\|_{\infty} 
 \bigg\|\int_0^t \nabla \div v (\tau,X(\tau,\cdot))\,ds\biggr\|_{q}\biggr)
 \bigg\|\exp\Big\{3\int_0^t |\nabla v(\tau,X(\tau,\cdot))|\,d\tau \Big\}\bigg\|_{m}\cdotp
\end{multline}

 To bound the last term,  we write that for all $\beta>0,$
$$ \int_0^t |\nabla v(\tau,X(\tau,\cdot)|\, d\tau  \leq \beta
  \int_0^t \frac{|\nabla v(\tau,X(\tau,\cdot))|^2}{\|\nabla^2 v(\tau,\cdot)\|^2_{2}}\,d\tau + \frac{1}{4\beta} \int_0^t \|\nabla^2 v(\tau,\cdot)\|^2_{2}\,d\tau.  $$

 Hence using the following   Jensen inequality,
 $$ \exp\Big\{\int_0^t\phi(s)\,ds\Bigr\}\leq \frac1t\int_0^t e^{t\phi(s)}\,ds, $$
we discover that 
 $$\displaylines{
 \int_{\T^2} \exp\Big\{3m \int_0^t |\nabla v(\tau,X(\tau,\cdot))|\,d\tau\Big\} \,dx \hfill\cr\hfill
 \leq \exp\bigg\{ \frac m{4\beta} \int_0^t \|\nabla^2 v(\tau,\cdot)\|^2_{2}\,d\tau\bigg\} 
 \frac{1}{t} \int_0^t \int_{\T^2} \exp\Big\{9m \beta  t 
 \frac{|\nabla v(\tau,X(\tau,\cdot)|^2}{\|\nabla^2 v(\tau,\cdot))\|^2_{2}} \Big\}\,dx\, d\tau.}
$$

 In the last integral we change coordinates and  get  
$$
\displaylines{
  \int_{\T^2} \exp\Big\{3m \int_0^t |\nabla v(\tau,X(\tau,\cdot))|\,d\tau\Big\} \,dx
 \leq \frac 1t \exp\Big\{ \frac m{4\beta} \int_0^t \|\nabla^2 v(\tau,\cdot)\|^2_{2}\,d\tau\Big\} \hfill\cr\hfill\times
 \biggl(\int_0^t \int_{\T^2} \exp\Big\{9m \beta  t \frac{|\nabla v(\tau,x))|^2}{\|\nabla^2 v(\tau,\cdot)\|^2_{2}} \Big\}dx\, ds\biggr)\exp\bigg(\int_0^t\|\div v\|_{\infty}\,d\tau\biggr)\cdotp}
$$

At this stage, to complete the proof,  it suffices to  apply the following Trudinger inequality (see e.g. \cite{Adams})
to $f=\nabla v$: there exist constants $\delta_0$ and $K$ such that for all $f$ in $H^1(\T^2),$
\begin{equation}\label{eq:T9}
 \int_{\T^2} \exp\bigg\{ \delta_0 \frac{|f(x)-\overline f|^{2}}{\|\nabla f\|^{2}_{2}} \bigg\} dx \leq K\with
 \overline f:=\frac{1}{|\T^2|} \int_{\T^2} f\,dx.
\end{equation}

Then, taking $\beta$ so small that $9m\beta t=\delta_0$, we end up with  
\begin{multline}\label{eq:T12}
 \int_{\T^2} \exp\Big\{3m \int_0^t |\nabla v(\tau,X(\tau,\cdot))|\,d\tau\Big\} \,dx \\
 \leq C \exp\bigg(\frac{9m t}{4\delta_0} \int_0^t\|\nabla^2 v(s,\cdot)\|^2_{2}\,ds\bigg)
\exp\bigg(\int_0^t\|\div v\|_{\infty}\,ds\biggr)\cdotp
\end{multline}

Combining  with \eqref{eq:T10} completes the proof 
of the proposition.
\end{proof}

%


\section{Linear systems  with variable coefficients}\label{s:maxreg}

Here we are concerned with the proof of maximal regularity  estimates for
the linear system 
\begin{equation}\label{eq:lamevar}
 \begin{array}{lcl}
  \rho u_t - \Delta u - \nu \nabla \div u = f &\hbox{in}& (0,T)\times\T^N,\\
  u|_{t=0}=u_0 &\hbox{in}& \T^N,
 \end{array}
\end{equation}
assuming only that  $\rho=\rho(t,x)$ is bounded by above and from below
(no time or space regularity whatsoever).
\smallbreak
In contrast with the previous section, we do not need the space dimension to be $2.$
As we want  to keep track of the dependency with respect to  $\nu$ for $\nu\to+\infty,$
we shall assume throughout that $\nu\geq0$ for simplicity.
\begin{thm}\label{th:lin}
 Let $T>0$ and assume that $\nu\geq0,$ 
 \begin{equation}\label{eq:nv}
  0<\rho_* \leq \rho(t,x)\leq \rho^* \mbox{ \ \ for \ \ } (t,x) \in  [0,T]\times\T^N.
 \end{equation}


 Then there exist positive constants $2_*,2^*$  depending only on $\rho_*$ and $\rho^*,$ with  $2_*<2<2^*,$  such 
 that for all $r \in (2_*,2^*)$ we have 
\begin{equation}
 \|u_t,\nabla^2 u,\nu \nabla \div u\|_{L_r( (0,T)\times\T^N)} \leq C(r,\rho_*,\rho^*)\bigl(\|f\|_{L_r( (0,T)\times\T^N))} + \|u_0\|_{W^{2-2/r}_r(\T^N)}\bigr).
\end{equation}
\end{thm}
\begin{proof}
First, we reduce the problem to the one with null initial data, solving
\begin{equation}\label{eq:homo}
 \begin{array}{rcl}
  \rho^*\bar u_t -\Delta \bar u - \nu \nabla \div \bar u =0 &\hbox{in}&  (0,T)\times\T^N,\\
  \bar u|_{t=0}=u_0 &\hbox{in}& \T^N.
 \end{array}
\end{equation}




Applying the divergence operator to the equation  yields 
\begin{equation*}\label{n6}
  \rho^*(\div \bar u)_t - (1+\nu)\Delta \div \bar u = 0.
\end{equation*}
Hence the basic maximal regularity theory  for the heat equation in the torus gives 
\begin{equation}\label{n7}
 (1+\nu)\|\nabla\div \bar u\|_{L_p( (0,T)\times\T^N)} \leq C\|\div u_0\|_{W^{1-2/p}_p(\T^N)}.
\end{equation}

Then we restate System \eqref{eq:homo} in the form
\begin{equation}\label{n8}
\rho^* \bar u_t -\Delta \bar u=\nu \nabla \div \bar u,
\end{equation}
and get
$$\begin{aligned}
 \|\bar u_t,\nabla^2 \bar u \|_{L_p(\T^N\times (0,T))} 
 &\leq K_p\bigl(\nu \|\nabla \div \bar u\|_{L_p( (0,T)\times\T^N)} + \|u_0\|_{W^{2-2/p}_p(\T^N)}\bigr)\\
 & \leq K_p\Bigl(\frac{\nu}{1+\nu}\Bigr)\|u_0\|_{W^{2-2/p}_p(\T^N)}.\end{aligned}
$$

Therefore, as  $\nu\geq0,$  we end up with 
\begin{equation}\label{eq:n1}
 \|\bar u_t,\nabla^2 \bar u, \nu\nabla\div\bar u \|_{L_p( (0,T)\times\T^N)}
  \leq K_p\|u_0\|_{W^{2-2/p}_p(\T^N)}.
\end{equation}

Next we look for $u$ in  the form
\begin{equation}\label{n0}
 u = w+\bar u,
\end{equation}
where  $w$ fulfills
\begin{equation}\label{n1}
 \rho w_t - \Delta w - \nu \nabla \div w = f + (\rho^*-\rho) \bar u_t=: g,\qquad
 w|_{t=0}=0.
\end{equation}

 Thanks to \eqref{eq:nv} and  \eqref{n1},  we have 
\begin{equation}\label{eq:n1b}
 \|g\|_{L_p( (0,T)\times\T^N)} \leq\|f\|_{L_p( (0,T)\times\T^N)}
  + K_p (\rho^*-\rho_*) \|u_0\|_{W^{2-2/p}_p(\T^N)}.
\end{equation}

Now, setting $h:=g+(\rho^*-\rho)w_t,$
System \eqref{n1} reduces to  the following one:
\begin{equation}\label{eq:n2}
 \begin{array}{rcl}
  \rho^* w_t -\Delta w - \nu \nabla \div w = h &\hbox{in}& (0,T)\times\T^N, \\[1ex]
  w|_{t=0} =0 &\hbox{in}& \T^N.
 \end{array}
\end{equation}

We claim that for all $p\in(1,\infty)$ we have
\begin{equation}
 \|\rho^* w_t\|_{L_p( (0,T)\times\T^N)} \leq C_p\|h\|_{L_p( (0,T)\times\T^N)}
\end{equation}
with $C_p\to1$ for $p\to2.$ 
\medbreak
Indeed, to see that $C_2=1,$ we just test the first equation of \eqref{eq:n2} by $w_t,$
which yields
$$\rho^*\|w_t\|_{L^2(\T^N)}^2+\frac12\frac d{dt}\biggl(\|\nabla w\|_{L^2}^2
+\nu\|\div v\|_{L^2}^2\biggr)= \int_{\T^N} h\,w_t\,dx.$$

Then for any fixed $p_0\in(1,+\infty)\setminus\{2\},$ the standard maximal regularity 
estimate reads 
$$
\|\rho^*w_t\|_{L_{p_0}( (0,T)\times\T^N)} \leq K_{p_0}\|h\|_{L_{p_0}( (0,T)\times\T^N)},
$$
and H\"older inequality gives us   for all $\theta\in[0,1],$ 
$$
\|z\|_{L_r( (0,T)\times\T^N)} \leq \|z\|_{L_2( (0,T)\times\T^N)}^{1-\theta} \|z\|_{L_{p_0}( (0,T)\times\T^N)}^\theta \ \hbox{ with }\   \frac{1}{r}= \frac{1-\theta}{2} + \frac{\theta}{p_0}\cdotp
$$
Therefore $C_p\leq C_{p_0}^\theta,$ whence $\limsup C_p\leq1$ for $p\to2$ (as  $\theta \to 0$). 
\medbreak
Now, remembering the definition of $h,$ we  write for all $p\in(1,\infty),$
$$\begin{aligned}
\|\rho^*w_t\|_{L_p( (0,T)\times\T^N)}&\leq C_p\bigl(\|g\|_{L_p( (0,T)\times\T^N)}
+\|(\rho^*-\rho)w_t\|_{L_p( (0,T)\times\T^N)}\bigr)\\
&\leq C_p\|g\|_{L_p( (0,T)\times\T^N)}+ C_p\biggl(1-\frac{\rho_*}{\rho^*}\biggr)\|\rho^*w_t\|_{L_p( (0,T)\times\T^N)}.
\end{aligned}$$
Therefore, if \footnote{Clearly, we just need that  $1-C_p(1-\frac{\rho_*}{\rho^*}) >0.$
 However taking  that slightly stronger condition allows to get a more explicit
 inequality.} 
\begin{equation}\label{eq:p}
1-C_p\biggl(1-\frac{\rho_*}{\rho^*}\biggr)\geq\frac12\frac{\rho_*}{\rho^*},
\end{equation}
then we end up with 
\begin{equation}\label{eq:n3}
\|\rho^*w_t\|_{L_p( (0,T)\times\T^N)}\leq \frac{2\rho^*C_p}{\rho_*}\: \|g\|_{L_p( (0,T)\times\T^N)}.
\end{equation} 
Let us emphasize that \eqref{eq:p} is fulfilled for $p$ close enough to $2,$
due to $C_p\to1$ for $p\to2.$\medbreak

It is now easy to complete the proof.  We take \eqref{eq:n2} in the form
$$
 \begin{array}{rcl}
   -\Delta w - \nu \nabla \div w = g - \rho w_t &\hbox{in}&  (0,T)\times\T^N, \\[1ex]
  w|_{t=0} =0 &\hbox{in}& \T^N.
 \end{array}$$
Then one  concludes as before that 
$$\begin{aligned}
 \|\nabla^2w,\nu\nabla\div w\|_{L_p( (0,T)\times\T^N)}&\leq 
 K_p  \|g-\rho w_t\|_{L_p( (0,T)\times\T^N)}\\
& \leq K_p\Bigl(  \|g\|_{L_p( (0,T)\times\T^N)} +C_{\rho_*,\rho^*}\|w_t\|_{L_p( (0,T)\times\T^N)}\Bigr)\cdotp
\end{aligned}
$$
Hence, putting together with \eqref{eq:n3} and assuming that $p$ is close enough to $2,$
\begin{equation}
 \|w_t,\nabla^2w,\nu\nabla\div w\|_{L_p( (0,T)\times\T^N)}\leq C_{\rho_*,\rho^*}
   \|g\|_{L_p( (0,T)\times\T^N)}.
   \end{equation} 
 Then combining with \eqref{eq:n1b} and \eqref{eq:n1}  completes the proof. \end{proof}


\section{Final bootstrap argument}

In what follows, we fix some $0<\rho_*<\rho^*$ and denote by $2_*$ and $2^*$
the corresponding Lebesgue exponents provided by Theorem \ref{th:lin}. 
We assume that the initial data  $(\rho_0,v_0)$ satisfies
all the requirements of Theorem \ref{th:1}

Take some time  $T$  such that $1\leq T \leq\nu$ (stronger conditions will appear below),
and  assume  that we have a solution $(\rho,v)$ to \eqref{CNS} on $[0,T]\times\T^2,$ fulfilling
the regularity properties of Theorem \ref{th:1}  for some $2<q<\min(2^*,4),$ and
\begin{equation}\label{eq:div}
\exp\biggl(\int_0^T\|\div v\|_{\infty}\,dt\biggr)\leq 2.
\end{equation}
Then it is clear that $\rho$   obeys
\begin{equation}\label{eq:boot1}
\rho_*\leq\rho \leq \rho^* \quad\hbox{on }\ [0,T]\times \T^2.
\end{equation}

For all $p\in[2,q],$  denote  $A_p(T):=\|\nabla \div v\|_{L_1(0,T;L_p(\T^2 ))}$ and assume that,
for some small enough constant $c_0>0,$ we have 
\begin{equation}\label{eq:smallA}
 A_q(T)\leq c_0.
\end{equation}
Obviously, if $Kc_0\leq\log2$ where $K$ stands for the norm of the embedding 
$\dot W^1_q(\T^2)\hookrightarrow L_\infty(\T^2),$ then \eqref{eq:div} is fulfilled.
We shall assume in addition that $c_0\rho^*\leq1.$
\medbreak

We are going to show that if \eqref{eq:smallA} is fulfilled then, for sufficiently large $\nu,$
all the norms of the solution are under control. Then, bootstrapping, this will justify  \eqref{eq:smallA} a posteriori.

\subsubsection*{Step 1. High order energy estimate for $v$}
Let  $E_0^2:=1+\|v_0\|_{H^1}^2+\|\rho_0-\bar\rho\|_{2}^2.$ By \eqref{eq:e0} we 
easily get, remembering that $\nu^{-1}T\leq1$, 
\begin{multline}\label{eq:H1}
\|v\|_{L_\infty(0,T;H^1)}^2+\nu\|\div v\|_{L_\infty(0,T;L_2)}^2+\|\rho-\bar\rho\|_{L_\infty(0,T;L_2)}^2\\+\int_0^T\bigl(\|\nabla v\|_{H^1}^2 +
\|v_t\|_{2}^2+ \nu\|\nabla\div v\|_{2}^2\bigr)dt\leq
 C e^{C\|v_0\|_{2}^4}  \bigl(E_0^2+\nu^{-1}T \|\nabla \rho\|_{L_\infty(0,T;L_2)}^2\bigr)\cdotp
\end{multline}

\subsubsection*{Step 2. Regularity estimates at $L_p$ level for the density}

From   Proposition  \ref{thm:transport}, we find that 
there exists an absolute constant $K$ such that 
for all $r\in[2,q),$ there exists some constant $C_r>0$ so that$$
\sup_{t\in[0,T]}\|\nabla \rho(t)\|_{r} \leq K\biggl(\bigl(\|\nabla\rho_0\|_{q}+\rho^*A_q(T)\bigr) \exp\biggl(C_rT\int_0^T\|\nabla^2 v\|_{2}^2\,dt\biggr)\biggr)\cdotp
$$
Hence,  bounding the last term according to \eqref{eq:H1}, and using \eqref{eq:smallA} and the definition of~$E_0,$
\begin{multline}\label{eq:boot2}
\sup_{t\in[0,T]}\|\nabla \rho(t)\|_{r}\\ \leq K\bigl(\|\nabla\rho_0\|_q+1\bigr)\exp(C_rE_0^2Te^{C\|v_0\|_{2}^4}) \exp\Bigl(C_r\nu^{-1}T^2 e^{C\|v_0\|_{2}^4}\|\nabla \rho\|_{L_\infty(0,T;L_2)}^2\Bigr)\cdotp
\end{multline}
Taking $r=2,$ we deduce that  if 
\begin{equation*}
C_2\nu^{-1}T^2 e^{C\|v_0\|_{2}^4}\|\nabla \rho\|_{L_\infty(0,T;L_2)}^2\leq \log2 ,
\end{equation*}
 then we have 
\begin{equation}\label{eq:nablarho}
\sup_{t\in[0,T]}\|\nabla \rho(t)\|_{2} \leq 2K\bigl(\|\nabla\rho_0\|_q+1\bigr) 
\exp(C_2E_0^2Te^{C\|v_0\|_{2}^4}\bigr)\cdotp
\end{equation}

Using an obvious connectivity argument, we conclude that  
 \eqref{eq:nablarho} holds true whenever
\begin{equation}\label{eq:nu}
\nu>\frac{4K^2C_2}{\log2}\bigl(\|\nabla\rho_0\|_q+1\bigr)^2 \exp(2C_2E_0^2Te^{C\|v_0\|_{2}^4}\bigr) T^2 e^{C\|v_0\|_{2}^4}.
\end{equation}

Reverting to \eqref{eq:H1}, we readily get,
 taking a larger constant $C$ if need be,
\begin{multline}\label{eq:H1b}
\|v\|_{L_\infty(0,T;H^1)}^2+\nu\|\div v\|_{L_\infty(0,T;L_2)}^2+\|\rho-\bar\rho\|_{L_\infty(0,T;L_2)}^2\\+\int_0^T\bigl(\|\nabla v\|_{H^1}^2
+\|v_t\|_{L_2}^2 + \nu\|\nabla\div v\|_{L_2}^2\bigr)dt\leq C e^{C\|v_0\|_{2}^4} E_0^2.
\end{multline}

Of course, combining \eqref{eq:nablarho} with \eqref{eq:boot2} ensures that for all $r\in[2,q),$ we have
\begin{equation}\label{eq:boot3}
\sup_{t\in[0,T]}
\|\nabla \rho(t)\|_{L_r} \leq K\bigl(\|\nabla\rho_0\|_q+1\bigr)\, \exp(C_rE_0^2Te^{C\|v_0\|_{2}^4}).
\end{equation}

\subsubsection*{Step 3. Maximal regularity at $L_p$ level for the velocity} 

We rewrite the velocity equation as follows:
$$\rho\d_tv-\Delta v-\nu\nabla\div v=-\nabla P-\rho v\cdot\nabla v.$$

Then Theorem \ref{th:lin} ensures that for all $p\in[2,q),$
\begin{equation}\label{eq:boot4}
V_p(T)\leq C_p\bigl(\|v_0\|_{W^{2-\frac2p}_{p}}+\|\nabla P+\rho v\cdot\nabla v\|_{L_p(0,T\times\T^2)}\bigr)
\end{equation}
with 
$V_p(T):=\|v\|_{L_\infty(0,T;W^{2-\frac2p}_{p})} +\|v_t,\nabla^2v,\nu\nabla\div v\|_{L_p(0,T\times\T^2)}.$
\medbreak
By H\"older inequality
$$\|v\cdot\nabla v\|_{L_p(0,T\times\T^2)}\leq T^{\frac1s}\|v\|_{L_\infty(0,T;L_s)}\|\nabla v\|_{L_4(0,T;L_4)}
\quad\hbox{with}\quad \frac1s+\frac14=\frac1p\cdotp $$
Hence using embedding and Inequality \eqref{eq:H1b},  
$$\|v\cdot\nabla v\|_{L_p(0,T\times\T^2)}\leq C T^{\frac1p-\frac14} E_0^2e^{C\|v_0\|_{2}^4},$$ 
and reverting to \eqref{eq:boot4} and using \eqref{eq:boot3} thus yields for some constant 
$C_P$ depending only on the pressure law,
\begin{equation}\label{eq:boot5}
V_p(T)\leq C_p\Bigl(\|v_0\|_{W^{2-\frac2p}_{p}}+C_PT^{\frac1p}\bigl(\|\nabla\rho_0\|_{q}+1\bigr)e^{CE_0^2Te^{C\|v_0\|_{2}^4}}
+ T^{\frac1p-\frac14} E_0^2e^{C\|v_0\|_{2}^4}\Bigr)\cdotp\end{equation}

\subsubsection*{Step 4. Regularity estimate at $L_q$ level for the density}

The standard estimate for transport equation with Lispchitz velocity field yields
$$ \sup_{t\leq T} \|\nabla \rho(t)\|_{q}\leq \bigl(\|\nabla \rho_0\|_{q} 
 + \rho^* A_q(T)\bigr) \exp\{\|\nabla v\|_{L_1(0,T;L_\infty)}\}\cdotp$$
 Hence, remembering \eqref{eq:smallA} and  using the embedding $\dot W^1_p(\T^2)\hookrightarrow L_\infty(\T^2)$ 
 to handle the last term, we get
$$ \sup_{t\leq T} \|\nabla \rho(t)\|_{q} \leq    \bigl(\|\nabla\rho_0\|_{q}+1\bigr) \exp\bigl\{CT^{\frac1{p'}}V_p(T)\bigr\}\cdotp$$
Then one can bound $V_p(T)$ according to \eqref{eq:boot5} and eventually get,
 \begin{equation}\label{eq:boot6}
 \sup_{t\leq T} \|\nabla \rho(t)\|_{q} \leq  \bigl(\|\nabla\rho_0\|_q+1\bigr) \exp\bigl\{T^{\frac1{p'}}I_0^p(T)\bigr\},
\end{equation}
with $I_0^p(T):= C_p\Bigl(\|v_0\|_{W^{2-\frac2p}_{p}}+C_PT^{\frac1p}\bigl(\|\nabla\rho_0\|_{q}+1\bigr)e^{CE_0^2Te^{C\|v_0\|_{2}^4}}\Bigr)\cdotp$

\subsubsection*{Step 5. Maximal regularity at $L_q$ level for the velocity}

Let us use again  Theorem \ref{th:lin}, but with Lebesgue exponent $q.$ We have 
\begin{equation}\label{eq:boot7}
V_q(T)\leq C_q\bigl(\|v_0\|_{W^{2-\frac2q}_{q}}+\|\nabla P\|_{L_q(0,T\times\T^2)}+\|\rho v\cdot\nabla v\|_{L_q(0,T\times\T^2)}\bigr)\cdotp
\end{equation}
The last term may be bounded as in \eqref{eq:boot5} (with $q$ instead of $p$), and the pressure
term may be handled thanks to \eqref{eq:boot6}. At the end we get
$$V_q(T)\leq C_q\Bigl(\|v_0\|_{W^{2-\frac2q}_{q}}+  
C_PT^{\frac1q} \bigl(\|\nabla\rho_0\|_{L_q}+1\bigr)\exp (T^{\frac1{q'}} I_0^q(T))\Bigr)\cdotp$$

\subsubsection*{Step 6. Final bootstrap} 

In order to complete the proof, it suffices to check that if $\nu$ is large enough then 
we do have \eqref{eq:smallA}. 
This is just a consequence of the fact that
$$A_q(T)\leq T^{\frac1{q'}}\|\nabla\div v\|_{L_q(0,T\times\T^2)} \leq \frac1\nu T^{\frac1{q'}}V_q(T).$$
Hence it suffices to choose $\nu$ fulfilling \eqref{eq:nu} and
\begin{equation*}\label{eq:nu2}
\nu\geq T^{\frac1{q'}} 
C_q\Bigl(\|v_0\|_{W^{2-\frac2q}_{q}}+  
C_PT^{\frac1q}\bigl(\|\nabla\rho_0\|_{L_q}+1\bigr) \exp (T^{\frac1{p'}} I_0^q(T))\Bigr)\cdotp
\end{equation*}


\section{The incompressible limit issue}

The aim of this section is to prove Theorem \ref{th:2}.
In what follows the time $T$ is fixed, and $\nu$ is larger than the threshold viscosity $\nu_0$
given  by Theorem \ref{th:1}. Throughout, we shall agree that $C_{0,T}$ denotes a `constant'
depending only on $T$ and on the norms of the initial data 
appearing in Theorem \ref{th:1}. 
Let us consider the corresponding solution $(\rho,v).$ Then 
 Inequality  \eqref{E0} already ensures that 
all  the terms with $\cQ v$ in \eqref{eq:conv} are bounded as required.
\smallbreak
In order to bound the other terms of \eqref{eq:conv}, 
it is convenient to  restate System \eqref{CNS} in terms of the divergence-free part $\cP v$ and potential part
$\cQ v$ of the velocity field $v,$
  and of the discrepancy $r:=\rho-\wt\rho$  between $\rho$ and the following `incompressible'  density $\wt\rho$
defined as the unique solution of the following transport equation:
\begin{equation}\label{x3}
 \wt\rho_t + \cP v \cdot \nabla \wt\rho=0, \qquad \wt\rho|_{t=0}=\rho_0.
\end{equation}

As  $r$ fulfills:
\begin{equation}\label{x4}
 r_t + \cP v \cdot \nabla r = -\div(\rho \cQ v), \qquad r|_{t=0}=0,
\end{equation}
 we have for all $t\in[0,T],$
\begin{equation}\label{x6}
\|r(t)\|_q\leq \int_0^t\bigl(\|\rho\,\div\cQ v\|_q + \|\cQ v\cdot\nabla\rho\|_q\bigr)\,d\tau.
\end{equation}
Now, we have
$$\| \cQ v \cdot \nabla \rho\|_{L_q(0,T\times\T^2)}\leq \|\cQ v\|_{L_q(0,T;L_\infty)}\|\nabla\rho\|_{L_\infty(0,T;L_q)}$$
and,  by virtue of Poincar\'e inequality, 
$$
\|\rho\,\div\cQ v\|_{L_q(0,T\times\T^2)}\leq C\rho^*\|\nabla\div\cQ v\|_{L_q(0,T\times\T^2)}.
$$
Therefore, taking advantage of Sobolev embedding and of Inequality  \eqref{E1},
we end up with 
\begin{equation}\label{x7}
 \sup_{0\leq t\leq T} \|r(t)\|_{q} \leq C_{0,T} \nu^{-1}.
\end{equation}

Next, we restate the equation $(\ref{CNS})_2$ as follows:
\begin{equation}\label{mm1}
 \wt\rho \cP v_t + \wt\rho \cP v \cdot \nabla \cP v - \Delta \cP v + \nabla Q +K=0\end{equation}
 $$\displaylines{\with  Q:=P-(1+\nu)\div v,\quad  
  K_1:=r \cP v_t,\quad K_2:= \rho \cQ v_t,\hfill\cr\hfill
K_3:=  r \cP v \cdot \nabla \cP v \ \hbox{\ and \ } K_4:=  \rho (\cQ v \cdot \nabla\cP v  +  v \cdot \nabla \cQ v).}
$$


Subtracting  \eqref{INS} from \eqref{mm1}, we get
\begin{equation}\label{dd1}
  \eta (\cP v -u)_t +  \eta u \cdot \nabla (\cP v -u) - \Delta (\cP v - u) + \nabla(Q-\Pi)+ K+L=0
 \end{equation}
  with $$
L:=(\wt\rho - \eta) \cP v_t + (\wt\rho - \eta)\cP v \cdot \nabla \cP v  +\eta(\cP v -u) \cdot \nabla \cP v.$$
Of course,  initially, we have
\begin{equation*}
 \cP v - u|_{t=0} =0, \qquad \qquad \wt\rho -  \eta |_{t=0}=0.
\end{equation*}

Now,  we test \eqref{dd1} by $\cP v -u$ getting, since $\div u=0,$
\begin{equation}\label{x10}
\frac12 \frac{d}{dt} \int_{\T^2} \eta |\cP v -u|^2\,dx +\int_{\T^2} |\nabla (\cP v - u)|^2 \,dx = \int_{\T^2} K\cdot (u-\cP v)\,dx + \int_{\T^2} L\cdot(u-\cP v)\,dx.
\end{equation}

To analyze the  terms of the left-hand side,  
we need some information coming from  the continuity equations. The difference of $\wt\rho$ and $\eta$ fulfills 
\begin{equation*}\label{dd2}
 (\wt\rho - \eta)_t + u \cdot \nabla (\wt\rho -\eta) = - (\cP v - u) \cdot \nabla \wt\rho.
\end{equation*}
Testing it by $(\wt\rho - \eta)$ and defining $q^*$ by $\frac{1}{q^*}+\frac{1}{q}=\frac 12,$
 we find that
$$ \sup_{t \leq T} \|(\wt\rho - \eta)(t)\|_2 \leq  \int_0^T \|\cP v -u\|_{q^*} \|\nabla\wt\rho\|_{q}\, dt.$$
As $\wt\rho$ satisfies \eqref{x3}, we have for all $t\in[0,T],$
$$ \|\nabla\wt\rho(t)\|_{q}\leq  \|\nabla\wt\rho_0\|_{q} \,e^{\int_0^t\|\nabla\cP v\|_\infty\,d\tau}\cdotp$$
Therefore, thanks to \eqref{E2} and Sobolev embedding, 
   \begin{equation}\label{xx10}
  \sup_{t \leq T} \|(\wt\rho - \eta)(t)\|_2 \leq 
  C_{0,T}\int_0^T \|\cP v -u\|_{q^*} \,dt.
\end{equation}
One can now estimate all the terms of the right-hand side  of \eqref{x10}. Regarding  the first term of $L,$ 
we have 
$$\begin{aligned}
\int_0^T\!\!\! \int_{\T^2} (\wt\rho - \eta) \cP v_t \cdot (\cP v -u) \,dx \,dt &\leq 
\int_0^T \|\wt\rho-\eta\|_2 \|\cP v_t\|_{q}\|\cP v - u\|_{q^*}\, dt \\
\leq C_{0,T}\biggl(\int_0^T\|\cP v&- u\|_{q^*}\,dt\biggr) \biggl(\int_0^T \|\cP v_t\|_{q}^2\,dt\biggr)^{1/2} \biggl(\int_0^T \|\cP v- u\|_{q^*}^2 \,dt\biggr)^{1/2}.
\end{aligned}
$$
Hence taking  $\theta \in (0,1)$ below according to Gagliardo-Nirenberg inequality,
and remembering that $q>2$ and that  $H^1(\T^2) \hookrightarrow L_m(\T^2)$ for all $m<\infty,$
 we get 
\begin{multline}\label{x11}
\int_0^T\!\!\! \int_{\T^2} (\wt\rho - \eta) \cP v_t \cdot (\cP v -u) \,dx\,dt 
\leq C_{0,T} \int_0^T \|\nabla(\cP v- u)\|_{2}^{2\theta}  \|\cP v- u\|_2^{2-2\theta}\,dt \\
\leq \frac18\int_0^T \|\nabla(\cP v- u)\|_{2}^{2}\,dt + C_{0,T} \int_0^T \|\cP v- u\|_2^2 \,dt.
\end{multline}
 Next, we write
 $$
\left|\int_{\T^2} (\wt\rho -\eta) (\cP v\cdot \nabla \cP v)\cdot(\cP v- u)\, dx \right| \leq 
\|\wt\rho-\eta\|_{2}\|\cP v\cdot\nabla\cP v\|_{q}\|\cP v-u\|_{q^*},
$$
hence, arguing exactly as above, 
$$
\left| \int_0^T \!\!\!\int_{\T^2} (\wt\rho - \eta) (\cP v \cdot \nabla \cP v)\cdot (\cP v- u)\, dx\,dt \right| \leq \frac18 \int_0^T \|\nabla(\cP v- u)\|_{2}^{2}\,dt + C_{0,T} \int_0^T\! \|\cP v- u\|_2^2 \,dt.
$$
Similarly, we have
$$\left| \int_0^T\!\!\! \int_{\T^2}  \eta \big( (\cP v- u) \cdot \nabla \cP v\big)\cdot (\cP v -u) \,dx\,dt \right| 
\leq \rho^*\int_0^T\|\nabla \cP v\|_\infty \|\cP v -u\|_2^2 \,dt.$$
As regards $K_1,$ we have, defining  $\wt q$  by $\frac{2}{q} + \frac{1}{\wt q} =1,$
$$\begin{aligned}
\left| \int_0^T\!\!\!\int_{\T^2} r \cP v_t \cdot(\cP v -u) \,dx\, dt \right| &\leq \int_0^T \|r\|_q \|\cP v\|_q \|\cP v - u\|_{\wt q} \,dt \\
&\leq \frac18\int_0^T \|\nabla(\cP v -u)\|_{2}^2 \,dt + C_{0,T} \int_0^T\|\cP v-u\|_2^2\,dt,
\end{aligned}$$
and for $K_2,$ one can write that 
\begin{equation*}
 \int_{\T^2} \rho \cQ v_t \cdot(\cP v - u)\, dx = \frac{d}{dt} \int_{\T^2} \rho \cQ v \cdot(\cP v -u) \,dx - \int_{\T^2} (\rho (\cP v -u))_t \cdot\cQ v \,dx.
\end{equation*}
For the last term, we have, using that $\rho_t=-\div(\rho v)$ and integrating by parts,
$$\begin{aligned}
 \int_{\T^2} (\rho (\cP v -u))_t \cdot\cQ v \,dx&= \int_{\T^2} \rho (\cP v -u)_t \cdot\cQ v \,dx
 + \int_{\T^2} \rho_t (\cP v -u) \cdot\cQ v \,dx\\
 &=\int_{\T^2} \rho (\cP v -u)_t \cdot\cQ v \,dx
 +\int_{\T^2} (\rho v)\cdot \bigl(\nabla(\cP v-u)\cdot\cQ v\bigr)dx\\
&\hspace{3cm}+\int_{\T^2} (\rho v)\cdot \bigl((\cP v-u)\cdot\nabla\cQ v\bigr)dx.\end{aligned}
$$
The first term is  of order $\nu^{-1}$ after time integration on $[0,T],$ since
 it may  be bounded by
$$\left| \int_{\T^2} (\rho (\cP v -u))_t \cdot\cQ v \,dx\right|\leq \rho^*\|\cQ v\|_2\bigl(\|\cP v_t\|_2+\|u_t\|_2\bigr).$$
For the second term, one may write
$$
\biggl|\int_{\T^2} (\rho v)\cdot \bigl(\nabla(\cP v-u)\cdot\cQ v\bigr)dx\biggr|
\leq\frac1{8} \int_{\T^2}\|\nabla(\cP v-u)\|_2^2\,dx+C(\rho^*)^2\|v\|_\infty^2\|\cQ v\|_2^2,
$$
and for the last one, we have
$$
\biggl|\int_{\T^2} (\rho v)\cdot \bigl((\cP v-u)\cdot\nabla\cQ v\bigr)dx\biggr|\leq
\rho^*\|v\|_\infty\|\cP v-u\|_2\|\nabla\cQ v\|_2.
$$
In the same way, we get 
 $$
 \left|  \int_0^T\!\!\! \int_{\T^2} (K_3+K_4)\cdot(\cP v-u) \,dx\,dt \right| \leq\int_0^T\|\cP v-u\|_{q^*}\bigl(
 \|\cQ v\|_q\|\nabla\cP v\|_2+\|\nabla\cQ v\|_q\|v\|_2\bigr)\,dt,
 $$ whence using \eqref{E1} and Poincar\'e inequality to handle the terms with $\cQ v,$
 \begin{equation*}
\left|  \int_0^T\!\!\! \int_{\T^2} (K_3+K_4)\cdot(\cP v-u) \,dx\,dt \right| \leq \frac18\int_0^T \|\cP v -u\|_{H^1}^2 \,dt + \nu^{-2} C_{0,T}.
\end{equation*}

Summing up, we return to \eqref{x10} and integrate, to  find 
\begin{multline*}
\rho_* \sup_{t\leq T} \|(\cP v -u)(t)\|_2^2 + \int_0^T\|\nabla (\cP v -u)\|^2_2 \,dt \\
\leq \sup_{t \leq  T} \bigg| \int_{\T^2} (\rho \cQ v)(t)\cdot (\cP v -u)(t)\, dx\bigg| +C_{0,T} \int_0^T \|\cP v-u\|_2^2\,dt + C_{0,T} \nu^{-1}.
\end{multline*}
But we see that
\begin{multline*}
\left| \int_{\T^2} \rho \cQ v \cdotp (\cP v  -u)\,dx\right| \leq \frac 12 \rho_* \|\cP v -u\|_2^2 + 
 C \|\cQ v\|^2_2 \leq \frac 12 \rho_* \|\cP v -u\|_2^2 + C_{0,T}\nu^{-1}.
\end{multline*}
So altogether, we get after using Gronwall lemma,
\begin{equation*}
 \sup_{t\leq T} \big(\|(\cP v -u)(t)\|_2^2 +\|(\wt\rho - \eta)(t)\|_2^2\big) + \int_0^T\|\nabla (\cP v -u)\|^2_2 \,dt \leq C_{0,T}\nu^{-1}.
\end{equation*}
Remembering \eqref{x7} and that  $\wt\rho-\eta= r+(\wt\rho- \eta)$ completes the proof  of Theorem \ref{th:2}.\qed
\bigbreak\noindent{\bf Acknowledgments.} 
The first author (R.D.) is partly supported by ANR-15-CE40-0011.
 The second author (P.B.M.) has been partly supported by National Science Centre grant
2014/14/M/ST1/00108 (Harmonia).

\end{document}